\documentclass[a4paper,11pt,reqno]{amsart}
\usepackage{amssymb}
\usepackage{amsmath}
\usepackage{amscd}
\usepackage{stmaryrd}
\usepackage{color}
\usepackage{graphicx}

\def\a{\alpha}
\def\b{\beta}
\def\c{\gamma}

\def\G{\Gamma}

\def\lg{\langle}
\def\rg{\rangle}
\def\ra{\rightarrow}

\def\.{\cdot}
\def\O{\Omega}
\def\nb{\nabla}
\def\l{\lambda}
\def\t{\tau}
\def\beq{\begin{equation}}
\def\eeq{\end{equation}}
\def\bi{\begin{enumerate}}
\def\ei{\end{enumerate}}
\def\bea{\begin{eqnarray}}
\def\eea{\end{eqnarray}}
\def\beas{\begin{eqnarray*}}
\def\eeas{\end{eqnarray*}}
\def\ba{\begin{array}}
\def\ea{\end{array}}
\def\x{\times}
\def\f{\varphi}

\def\o{\omega}

\def\e{\varepsilon}
\def\L{\Lambda}
\def\k{\kappa}
\def\s{\sigma}

\def\r{\end{proof}}

\def\ot{\otimes}

\def\ug{\mathfrak{u}}

\def\g{\mathfrak{g}}

\def\G{\Gamma}

\def\LRA{\Leftrightarrow}
\def\RA{\Rightarrow}

\def\cp{\dot\gamma}

\def \R{\mathbb{R}}

\def \Z{\mathbb{Z}}
\def \C{\mathbb{C}}
\def \T{\mathbb{T}}
\def \H{\mathbf{H}}

\def\End{{\rm End}}

\def\tl{\tilde}

\def\dl{{\delta}}

\def\we{\wedge}
\def\d{{\partial}}
\def\db{\overline\partial}

\def\be{\begin{equation}}

\def\ee{\end{equation}}

\def\tr{\mathrm{tr}}

\def\tsl{\widetilde{SL}(2,\mathbb{R})}
\def\sm{\smallsetminus}

\newtheorem{ede}{Definition}[section]

\newtheorem{cor}[ede]{Corollary}
\newtheorem{prop}[ede]{Proposition}
\newtheorem{lem}[ede]{Lemma}

\newtheorem{thm}[ede]{Theorem}

\newtheorem{lema}[ede]{Lemma}

\newtheorem{rem}[ede]{Remark}

\newtheorem{defi}[ede]{Definition}
\theoremstyle{definition}
\newtheorem{definition}[ede]{Remark}
\def\obs{\begin{definition}}
\def\eobs{\end{definition}}

\def\J{\mathcal{J}}
\def\op{\oplus}

\title{On the metric structure of some non-K\"ahler complex threefolds}
\author{Florin Alexandru Belgun}

\address{Florin A. Belgun\\ Fachbereich Mathematik\\ Universit\"at
  Hamburg\\Bundessstr. 55, Zi. 214\\ 20146 Hamburg}
 \email{florin.belgun@math.uni-hamburg.de}

\begin{document}

\begin{abstract}
We introduce a class of hermitian metrics with {\em Lee potential},
that generalize the notion of l.c.K. metrics with potential introduced
in \cite{ov} and show that in the classical examples of Calabi and Eckmann of complex structures
on $S^{2p+1}\x S^{2q+1}$, the corresponding hermitian metrics are of this
type. These examples satisfy, actually, a stronger
differential condition, that we call {\em generalized Calabi-Eckmann},
condition that is satisfied also by the {\em Vaisman} metrics
(previously also refered to as {\em generalized Hopf
  manifolds}). This condition means that, in addition to being with Lee
potential, the torsion of the {\em
  characteristic} (or Bismut) connection is parallel. We give a local
geometric characterization of these generalized Calabi-Eckmann
metrics, and, in the case of a compact threefold, we give detailed
informations about their global structure. More precisely, the cases
which can not be reduced to Vaisman structures can be obtained by
deformation of locally homogenous hermitian manifolds that can be
described explicitly.

 \end{abstract}




\vspace*{-1cm}
\maketitle

\section{Introduction}
The non-K\"ahlerianity of a hermitian metric is measured by the
exterior  differential of the K\"ahler form, which decomposes as
\be\label{lee} d\o=-2\theta\we\o+\Omega_0,\ee
where $\theta\in\L^1M$ is the {\em Lee form} of $\o$, and $\Omega_0$
the {\em trace-free} part of $d\o$. In complex dimensions greater than
$2$, the vanishing of $\Omega_0$ implies that $\theta$ is closed or,
equivalently, that the hermitian metric is {\em locally conformally
  K\"ahler} (l.c.K.), and this kind of metrics have been studied
for a long time by many specialists \cite{vais}, \cite{do},
\cite{lck}, \cite{ov}, 
 to name just a few.

In fact, the first examples of non-K\"ahler compact complex manifolds were
given by the Hopf manifolds $\C^n\sm\{0\}$ divided by a nontrivial
linear diagonal contraction, which are of this hermitian
type. Moreover, the standard flat metric on $\C^n$
induces, by renormalization, a {\em Vaisman} (formerly: {\em
  generalized Hopf}) metric on the quotient,
i.e. it is not only l.c.K., but its Lee form is not just closed, but
parallel as well. This lead to the former name of the structure, name which was
proven inappropriate in \cite{lck}, where examples of Hopf surfaces not
admitting such a Vaisman metric were given. Also in \cite{lck}, the stability
to small deformations of the class of Vaisman, or even l.c.K.
manifolds was disproven.

However, in \cite{ov}, Ornea and Verbitsky introduced the concept of
{\em l.c.K. metrics with potential}, or l.c.K.p., which form a subclass of
l.c.K. metrics containing the Vaisman structures, subclass which turns
out to be stable to deformations \cite{ov}. This class of metrics can
be characterized by the fact that its K\"ahler form is determined by
the $1$-form $\theta$ through the following relation: 
\be\label{lckp} \o=c\left( \theta\we
  J\theta+\frac{1}{2}d(J\theta)\right),\ \mbox{
  and } d\theta=0,\ee 
where $\theta$ is the (closed) Lee form and $c>0$ is a constant, hence
 $\o$ is l.c.K. Here, and throughout the paper, $\o:=g(J\cdot,\cdot)$
is the K\"ahler form of the metric $g$, $J$ is the complex structure,
and $J\theta:=-\theta\circ J$.  

If we consider the exterior derivative of (\ref{lckp}), and we no longer assume that the (real) Lee form $\theta$ is closed,
but only that $\db\theta^{(0,1)}=0$ (equivalent to the fact that
$d\theta$ or $d(J\theta)$ is a $(1,1)$--form),  then we obtain 
\be\label{lp}d\o=c\left(-\theta\we d(J\theta)+d\theta\we
  (J\theta)\right)\ \mbox{ and } \db\theta^{(0,1)}=0,\ee
and we call the
corresponding metric a {\em metric with Lee potential} (in short
LP). Note that, if (\ref{lp}) is satisfied for $\theta$, then a similar relation
holds (with some constant $c'$) for any non-zero linear combination of $\theta$
and $J\theta$. In (\ref{lp}), therefore, the Lee form does not play a
priviledged role, as in (\ref{lckp}), but only the ``complex line''
determined by $\theta$ occurs as the set of potentials for $d\o$.  
\smallskip

Note that from a topological point of view, the {\em
  non-K\"ahlerianity} of an l.c.K. metric (or, more 
particularly, l.c.K.p. or even Vaisman) lies in the first
cohomology group of the manifold. More precisely, any such metric on a
simply-connected manifold is automatically K\"ahler. 

On the other hand, in 1953, Calabi and Eckmann gave an example of
compact, non-K\"ahler 
complex homogeneous structure on a product of odd-dimensional
spheres, \cite{ce}. This generalizes the classical example of Hopf manifolds, in
which case one of the factors is a circle, but, if the dimension of
both factors is at least 3, it produces compact, simply-connected
examples of non-K\"ahler manifolds. 

The first remark that we make in this paper is that the Calabi-Eckmann
manifolds, with their standard product metric, satisfy the equations
(\ref{lp}), and are, hence, metrics with Lee potential. 
On the other hand, they generalize the Vaisman structures in the
following way:
\begin{defi}\label{gce} An LP hermitian metric that has a non-zero Lee
  form and such that the torsion of its {\em characteristic (or Bismut)
  connection} is parallel is called {\em generalized Calabi-Eckmann}
(or {\em GCE}).\end{defi}

Recall that the characteristic connection of a hermitian metric is unique
with the property that it preserves the hermitian structure and that its
torsion tensor is totally skew-symmetric (a $3$-form $T$, of type
$(2,1)+(1,2)$). By contraction with the K\"ahler form we obtain $J$
times the Lee form which is, in our GCE case, parallel {\em with
  respect to the characteristic connection}. Only if the Lee form is
closed is this form parallel also with respect to the Levi-Civita
connection (hence the Vaisman manifolds form a subclass of GCE manifolds). 

We obtain thus that the GCE manifolds form a subclass of LP manifolds,
but, in the case of complex threefolds, they can be deformed to become
either Vaisman manifolds, local products of a Vasiman surface with a Riemann
surface, or local products of {\em Sasakian} manifolds, see section
\ref{local}).



The latter (generic) case will be described in detail; we show, in
particular, that such a local Sasakian product is always a deformation
of a locally homogeneous one, and the underlying manifolds are
$3$--dimensional analogues to the non-K\"ahler elliptic
surfaces and Hopf surfaces, Theorem
\ref{th1}. 

Because the GCE structures are related to Sasakian geometry, as
much as the Vaisman structures are, it is not surprising that the
results in this article are somewhat similar to the ones in
\cite{lck}. However, the details are only indirectly related, and -- more
important --, we do not have a classification of compact complex
threefolds to use, as we did for surfaces \cite{lck}. Therefore,
although the fundamental group of a GCE threefold is not fully
understood, most topological and hermitian properties of these complex
manifolds can be described in detail.

The paper is organized as follows: In the first section we recall some
facts about hermitian geometry and the characteristic connection. In
the second, we study the case of parallel characteristic torsion and
obtain a local geometric characterization of GCE manifolds in all
dimensions, extending the results of \cite{afs}, \cite{sch}. 
In the third, we recall elements of  Sasakian geometry and facts
about l.c.K.p. and Vaisman geometry. Finally, we prove Theorem
\ref{th1} by examining the various cases that can occur on a compact
GCE threefold.

{\bf Acknowledgements: } The author is grateful to the SFB 676,
member of which he was during the research presented here, and in
particular to Vicente Cort\'es for many fruitful discussions.

\section{Hermitian geometry: basics and the characteristic connection}
Let $(M,J)$ be a complex manifold, i.e. $J$ is an endomorphism of
$TM$ of square $-I$ and {\em integrable}, i.e. its Nijenhuis tensor
$N^J\in\L^2M\ot TM$, defined by
\be\label{nij}4N^J(X,Y):=[JX,JY]-J[JX,Y]-J[X,JY]-[X,Y]\ee
vanishes identically. Let $g$ be a hermitian metric,
i.e. $g$ is a (positive definite) Riemannian metric, such that
$g(J\cdot,J\cdot)=g$. In fact, a hermitian metric is a reduction to
$U(m)$ (if $\dim_\C M=m$) of the structure group of $TM$. Let $\o:=g(J\cdot,\cdot)$ be the K\"ahler form
of the hermitian metric $g$. It is a $(1,1)$-form, and the space of
$k$-forms on $M$ decomposes in the following $U(m)$-irreducible
components:
$$\L^kM\ot\C=\displaystyle{\op}_{p=0}^k \L^{p,k-p}M, \mbox { and }\\
\L^{p,k-p}M=\L^{p,k-p}_0 M\op \L^{p-1,k-p-1}M,\ \forall
p=1,\dots,k-1,$$
where $\L^{p,q}_0M$, for $p,q>0$, is the kernel of the contraction with $\o$,
i.e. of the map from $\L^{p,q}M$ to $\L^{p-1,q-1}M$ defined by
$$\a\longmapsto \sum_{i=1}^n\a(e_i,Je_i,\dots),$$
where $\{e_i\}$ is a hermitian-orthogonal basis of $TM$.

Note that the real part of a $(p,q)$-form is equally the real part of
its conjugate, which is a $(q,p)$-form, hence we denote by
$\L^{(p,q)+(q,p)}M$ the real part of the sum of the vector spaces
$\L^{p,q}M$ and $\L^{q,p}M$. For $p=q$, our convention is to use {\em
  real} forms, so we denote by $\L^{p,p}M$ the space of {\em
  real} $(p,p)$-forms.

The operator $J\in\End(TM)$ can be considered at the same time as an element in
the Lie group $U(n)$ and in its Lie algebra $\ug(n)$. We denote by
$J.\a$ the Lie algebra action on $J$ on the tensor $\a$, and by $\J\a$
the Lie group action of $J$ on the same tensor. If $\a$ is a $k$-form,
then 
$$(\J\a)(X_1,\dots,X_k)=(-1)^k\a(JX_1,\dots,JX_k),$$
$$(J.\a)(X_1,\dots,X_k)=-\sum_{i=1}^k\a(X_1,\dots,JX_i,\dots,X_k).$$
We have then
\be\ba{lllll}\a\in\L^1M&\RA&\J\a=J.a&&\\
\a\in\L^{1,1}M&\LRA&\J\a=\a&\LRA&J.\a=0\\
\a\in\L^{2,0+0,2}M&\LRA&\J\a=-\a&\LRA&J.\a\in\L^{2,0+0,2}M\\
\a\in\L^{2,1+1,2}M&\LRA&\J\a=J.\a&&\ea\ee

We are next interested in the Hodge $*$ operator and its relations to
$J$. Recall:
\begin{defi} Let $(M,g)$ be an oriented Riemannian manifold and let
  $v_g\in\L^nM$ be its canonical volume form and denote by
  $\lg,.,.\rg$ the induced scalar product on any tensor space of $M$. The {\em Hoge star operator} $*=*_g:\l^pM\ra\L^{n-p}M$ is chracterized by the property
$$\forall\beta\in \L^{n-p}M,\ \lg*\a,\b\rg v_g=\a\we\b.$$
\end{defi}
Equivalently, if $e_1,\dots,e_n$ is a $g$--orthonormal basis of $TM$, 
\be\label{star} *(e_1\we\dots\we e_p)=e_{p+1}\we\dots\we e_n.\ee
On an almost hermitian manifold, we can choose the basis to be adapted
to the complex structure, i.e., $Je_{2i-i}=e_{2i}$, $\forall
i=1,\dots,m$, and we have
\begin{prop}\label{ide} On an almost hermitian manifold with K\"ahler
  form $\o$, we have the following identities:
\bea v_g&=&\frac{\o^m}{m!},\\
*\a&=&J.\a\we\frac{\o^{m-1}}{(m-1)!}\ \forall\a\in\L^1M,\\
*\frac{\o^k}{k!}&=&\frac{\o^{m-k}}{(m-k)!},\\
*\a&=&-\a\we\frac{\o^{m-2}}{(m-2)!},\ \forall\a\in\L^{1,1}_0M,\\
*\a&=&\a\we\frac{\o^{m-2}}{(m-2)!},\ \forall\a\in\L^{2,0+0,2}M,\\
*(\a\we\o)&=&J.\a\we\frac{\o^{m-2}}{(m-2)!}\ \forall\a\in\L^1M,\\
*\a&=&\J\a\we\frac{\o^{m-3}}{(m-3)!}\ \forall\a\in\L^{2,1+1,2}_0M.\eea
\end{prop}
\begin{proof} We choose a hermitian basis
  $e_1,Je_1,\dots,e_m,Je_m$ of $TM$ and use (\ref{star}) to
  compute the image through $*$ of some generators of the corresponding
  form spaces; indeed, we denote $e_i\we J e_i$ by $\o_i$ and have:
$$\o=\sum_{i=1}^m\o_i,\ \o^k=\sum_{1\le i_1\dots i_k\le
  m}\frac{m!}{(m-k)!}\o_{i_1}\we\dots\we\o_{i_k}.$$
On the other hand, the spaces of forms considered in the Proposition
 have the following sets of generators:
\bea \{e_i,J e_i\}_i &\mbox{generate}& \L^1M,\\
\{\o_i-\o_j,e_i\we e_j+Je_i\we Je_j\}_{i\ne j} &\mbox{generate}&
\L^{1,1}_0M,\\ 
\{e_i\we e_j-Je_i\we Je_j\}_{i\ne j}&\mbox{generate}&
\L^{2,0+0,2}M,\eea
and use that the map
$$\L^1M\ot\L^{1,1}M\ra\L^{2,1+1,2}M,\ \theta\ot\a\mapsto\theta\we\a$$
is surjective, in order to get that
$$\{e_i\we(\o_j-\o_k),Je_i\we(\o_j-\o_k),e_i\we e_j\we e_k+e_i\we
Je_j\we Je_k, Je_i\we e_j\we e_k+Je_i\we
Je_j\we Je_k\},$$
with $i,j,k\in\{1,\dots,m\}$ distinct, generate $\L^{2,1+1,2}_0M$.

The claimed identities follow by straightforward computations.
\end{proof}

The {\em Lee form} $\theta$ of $\o$ is the unique $1$-form that
satisfies one of the following equivalent equations:
\be\label{dec}d\o=-2\theta\we\o+\O_0,\ \mbox{with }
\O_0\in\L^{(2,1)+(1,2)}_0,\ \Longleftrightarrow \ 
\theta =-\frac{1}{2(m-1)}J\dl^g\o.\ee
(The normalization factor $-2$ in the first equation is consistent
with our convention for the Sasakian structures -- see below.)
Here, and below, the {\em codifferential} is defined, on the
even-dimensional manifold $M$, by the usual
formula $\dl:=-*d*$.

The following result is classical:
\begin{prop} In the decomposition above, if $\dim_\R M=4$, then
  $\O_0=0$; if $\dim_\R M\ge 6$ and $\O_0=0$, then $\theta$ is closed
  and the metric is l.c.K. (i.e., for each point $x\in M$, there
  exists a K\"ahler metric $g_x$, defined on an open set $U_x$ containing $x$,
  which is conformally equivalent to $g$: $g_x=e^{f}g$, for some
  function $f:U_x\ra\R$.) \end{prop}
The proof is based on the fact that the wedge product with $\o$, 
$$\o\we\cdot: \L^{(p,q)}M\ra \L^{(p+1,q+1)}$$ 
is injective iff $p,q<\dim_\C M$.

A theorem by Gauduchon \cite{pg} implies that, on a compact
l.c.K. manifold, there is a unique metric (the {\em standard}, or the
{\em Gauduchon} metric), conformally equivalent with the original one,
for which the Lee form is harmonic. 

A special class of Gauduchon metrics consists of the {\em Vaisman
  structures}, for which the Lee form is non-zero and parallel. Formerly called
{\em generalized Hopf manifolds}, the corresponding complex manifolds
admit a non-vanishing holomorphic vector field (whose real part is the
metric dual to the Lee form), which is a strong topological
condition implying the vanishing of the Euler characteristic
$\chi(M)$. Another necessary condition for the existence of a
non-globally conformally K\"ahler l.c.K. metric is a non-zero first de
Rham cohomology vector space, since adding $df$ to the Lee form
corresponds to multiplying the metric with $e^{-2f}$.

In fact, for $M$ a compact complex surface, the first Betti number is
even iff the surface admits a K\"ahler metric. The author classified
the l.c.K. complex surfaces with $\chi(M)=0$, and also the Vaisman
structures on them (when they exist) in \cite{lck}. A conclusion of
this series of results is that neither the class of Vaisman surfaces,
nor the larger class of l.c.K. surfaces is stable by small
deformations. In \cite{ov}, Ornea and Verbitsky introduced an
intermediate class, of {\em l.c.k. metrics with potential}, for which
the K\"ahler form is determined by the (closed) Lee form:
$$\o=c\left(\theta\we J\theta +\frac{1}{2}d (J\theta)\right),\
c>0\mbox{ constant}\mbox{ and } \
d\theta=0.$$

\begin{defi}\label{deflp} A hermitian metric with K\"ahler form $\o$
  and Lee form $\theta$ is a {\em metric with Lee potential} (LP) iff 
$$d\o=c\left(d\theta\we J\theta -\theta\we d (J\theta)\right),\
c>0\mbox{ constant},\ \mbox{ and } \ \db\theta^{(0,1)}=0,$$
where $\theta^{(0,1)}$ is the $(0,1)$-part of the real form $\theta$:
$$\theta^{(0,1)}:=\frac{1}{2}\left(\theta+iJ\theta\right).$$
\end{defi}
Note that the condition $\db\theta^{(0,1)}=0$ is equivalent to 
$d\theta$ (and also $d(J\theta)$) being of type $(1,1)$. 

The LP metrics correspond to a special form of a refined version of
the decomposition (\ref{dec}). Note that, although $\O_0=0$ in
(\ref{dec}) implies that $d\theta=0$, the converse is 
not true. On the other hand, the $\db$-closure of the $(0,1)$-form
$\theta^{(0,1)}$ is a more general fact than the closure of some
linear combitation of $\theta$ and $J\theta$, and the
exactness of $\theta^{(0,1)}$ is equivalent to the vanishing of the
Dolbeault class $[\theta^{(0,1)}]\in H^{(0,1)}M$.

Before showing that the Calabi-Eckmann complex structures on
$S^{2p+1}\x S^{2q+1}$, $p,q>0$ admit a LP metric (the standard product
metric, or some straightforward modification of it), we need to recall
some basic facts of {\em Sasakian geometry}.

\begin{defi} A manifold $(N^{2n+1},g,\xi)$ is {\em Sasakian} iff $g$
  is a Riemannian metric, $\xi$
  is a unit Killing vector field (the {\em Reeb field}) whose
  covariant derivative is a complex structure on $H:=\xi^\perp$, which
  is {\em integrable} in the $CR$ sense.\end{defi}
Recall that a $CR$ manifold is an odd-dimensional manifold $N^{2m+1}$ endowed
with a distribution of hyperplanes $H$ which is a {\em contact
  structure} (i.e., $\eta\we (d\eta)^n\ne 0$ for any $1$--form $\eta$
whose kernel is $H$ -- $\eta$ is then called a {\em contact form}), and
$J:H\ra H$ is a complex structure on $H$ for which $d\eta|_H$ is of type
$(1,1)$ for every contact form $\eta$. The $CR$ structure is
{\em integrable} iff the Nijenhuis tensor $N^J\in \mbox{Hom}\left(\L^2
  H,H\right)$ of $J$, defined by (\ref{nij}), vanishes identically.

It is well-known that the round metric of an odd-dimensional sphere
$S^{2n+1}$ is
Sasakian, and the corresponding Reeb vector field generates the circle
action on $S^{2n+1}$ defined by the of multiplication of an element of
$S^{2n+1}\subset \C^{n+1}$ with a complex number of norm $1$, and the
basic Calabi-Eckmann structure on $S^{2p+1}\x S^{2q+1}$ is very simple
to describe in hermitian terms:

\begin{prop}\label{bas} Let $(N_1,g_1,\xi_1)$ and $(N_2,g_2,\xi_2)$` be two
  Sasakian manifolds of dimensions $2n_1+1$, resp. `$2n_2+1$, and let
  $(H_i,J_i)$ be their $CR$ structures. Then the product metric
  $g:=g_1+g_2$ is hermitian with respect to the following almost
  complex structure, which is actually integrable:
\bea JX&:=&J_iX,\ X\in H_i;\\
J\xi_1&:=&\xi_2.\eea\end{prop}
The proof is straightforward.

Recall that the product of a Sasakian manifold with a circle is
Vaisman, and we see that this is the particular case of the above
Proposition, where one of the
Sasakian manifolds is $1$--dimensional. All {\em Sasakian
  automorphisms} (i.e., the isometries that preserve the Reeb field)
of the factors become thus hermitian isometries of $(M,J,g)$, in
particular the standard Calabi-Eckmann complex manifold  $S^{2p+1}\x
S^{2q+1}$ is homogeneous, and its hermitian automorphism group is
$U(p+1)\x U(q+1)$.

In fact, Calabi and Eckmann gave a family of complex structures
on $S^{2p+1}\x S^{2q+1}$, depending on one parameter
$\a\in\H:=\{\a\in\C\ | \ \Im\a>0\}$, by
setting in the Proposition \ref{bas}
$$J_\a\xi_1:=\mbox{Re}(\a)\xi_1+\mbox{Im}(\a)\xi_2$$
instead of $J\xi_1=\xi_2$, and extending it by linearity and such that
$J_\a^2=-I$. For these complex structures (also integrable), the
product metric is not hermitian any more, but a straightforward
modification 
\be\label{ga}\ba{rclll}g_\a(\xi_i,X)&:=&0,&&X\in H_1\op H_2\\
g_\a(X,Y)&:=&g(X,Y),&&X,Y\in H_1\op H_2\\
g_\a(\xi_1,\xi_1)&:=&g(J_\a\xi_1,J_\a\xi_1)&:=&1\\
g_\a(\xi_1,J_\a\xi_1)&:=&0\ea\ee
is, and the hermitian automorphism group doesn't change. 

These more general complex structures, and their associated hermitian
metrics, will be locally characterized in the next section.

\section{Parallel characteristic torsion. Generalized Calabi-Eckmann
  structures}\label{local} 

We intend to (locally) characterize the Calabi-Eckmann complex and
hermitian structures by a
differential-geometric condition on the hermitian manifold
$(M,J,g)$. Recall that, if $(M,J,g)$ is Vaisman, then its Lee form,
being paralel, generates a local product structure (by the
decomposition Theorem of de Rham) which is easily checked to be locally
isomorphic to the construction of Proposition \ref{bas} (for $n_2=0$).

For $m_1,m_2>0$, we do not get any parallel $1$--form and, if $\a\in\H$ is
 generic, there is no parallel distribution in the Calabi-Eckmann
 manifold $(S^{2p+1}\x S^{2q+1},J_\a,g_\a)$ for the {\em Levi-Civita
   connection}.

The idea is then to use another canonical connection on this manifold,
one that respects both the metric and the complex structure: the
{\em characteristic} connection. This notion arises in a more
general context than the one of hermitian geometry:

\begin{defi} Let $M$ be a $n$-dimensional manifold with a pseudo-Riemannian
  $G$-structure on it, i.e. a reduction to $G$ of the structure group
  of the orthogonal frame bundle of some pseudo-Riemannian metric $g$,
  for a given representation $\rho:G\ra O(\R^n,g_0)$. A connection
  $\nb$ on $M$ is called {\em characteristic} iff its torsion
  $T^\nb\in\L^2M\ot TM\stackrel{g}{\simeq} \L^2M\ot \L^1M$ is totally 
  skew-symmetric, i.e., $T^\nb\in\L^3M$.
\end{defi}

Note that, in order to identify $TM$ with its dual, $G$ needs to
preserve some non-degenerate bilinear form $g_0$ on $\R^n$, and the
Lie algebra of $G$ consists -- under this identification -- of
skew-symmetric bilinear forms iff $g_0$ is symmetric. Therefore the
restriction to pseudo-Riemannian $G$-structures is necessary.

Because the map 
$$\nb\longmapsto T^\nb,$$
associating to a pseudo-Riemannian connection its torsion is
injective, the set of characteristic connections, if non-empty, is an
affine space modelled on the space of sections of
$$\T:=T^*M\ot\g(M)\cap \L^2M\ot T^*M\subset \L^3M,$$
where $\g(M)\subset \L^2M$ is (isomorphic to)  the adjoint bundle of
the $G$-structure. Note that the first term on the left hand side is
skew-symmetric in the last 2 arguments and the second term in the
first two. For example, if $\g=\mathfrak{so}(p,q)$, then $\T=\L^3M$,
and if $\g=\mathfrak{u}(p,q)$, then $\T=0$.

Therefore, if an almost pseudo-hermitian manifold admits a
characteristic connection, then it is unique. The torsion of this
connection is then the $(2,1)+(1,2)$--form
$-\J d\o=-J.d\o$ (see the conventions in the previous section). 
Indeed, the differential of the $\nb$--parallel form
$\o$ can be computed in term of the torsion 
$$T(X,Y)=\nb_XY-\nb_YX-[X,Y]$$
of $\nb$ by the formula
\be\label{dom}d\o(X,Y,Z)=\o(T(X,Y),Z)+\o(T(Y,Z),X)+\o(T(Z,X),Y).\ee
But $\o(T(X,Y),Z)=-g(T(X,Y),JZ)=-T(X,Y,JZ)$, and $T$ is
skew-symmetric, hence
$$d\o=J.T=\J T \ \Longleftrightarrow \ T=-\J d\o.$$



Coming back to the Calabi-Eckmann construction, or, more generally, to
the one in Proposition \ref{bas} (possibly with the hermitian
structure $(J_\a,g_\a)$), we will show that the characteristic torsion
is parallel:

\begin{prop}\label{ce-par} 
In the Sasakian product of Proposition \ref{bas}, and also for the
modified hermitian structure $(\ref{ga})$, the hermitian
structure has parallel characteristic torsion.
\end{prop}
\begin{proof}
Let us first remark that, if a $G$-structure has a unique
characteristic connection $\nb$ with parallel torsion $T$, any
$\nb$-parallel $1$-form $\eta$ is closed iff it is parallel for the
Levi-Civita connection, because $d\eta$ is a multiple of
$T(\eta^\sharp)$, where $\eta^\sharp$ is the dual vector to $\eta$.

Then, let us consider a Sasakian manifold $(N,g,\xi)$. Because the
Vaisman manifold $N\x S^1$ has parallel characteristic torsion
$J\theta\we\o$, then the Lee form is parallel w.r.t. both the
Levi-Civita connection and the characteristic connection
$\nb$. Because $\theta$ (which is a non-zero multiple of the length
form on the $S^1$ factor is $\nb$-parallel, then $\nb$ induces a
characteristic connection w.r.t. the Sasakian structure on the factor $N$ as
well. Conversely, the product of the characteristic connections of two
factors is a characteristic connection for the product
$G$-structure. 

From here we infere that the characteristic connection of a Sasakian
manifold is unique, that its torsion is parallel, and that the product
of two Sasakian manifolds has parallel characteristic torsion as
well.

The fact that the modifications $(J_\a,g_\a)$ of the basic Sasakian product of
Proposition (\ref{bas}) are LP and have parallel characteristic
torsion will be a consequence of Proposition \ref{pardef} below.
\end{proof}

We define thus:
\begin{defi} A hermitian manifold $(M,g,J)$ with non-zero Lee
  potential and with parallel characteristic connection is called {\em
    generalized Calabi-Eckmann}, in short GCE.\end{defi}

Let $(M,g,J,\o)$ be a hermitian manifold with characteristic connection
$\nb$, and suppose that its torsion $T=-\J d\o$ is a
$\nb$-parallel $(2,1)+(1,2)$--form. Later, we will focus on the GCE
case, but for now, we just assume that the Lee form $\theta$ of $\o$
is non-zero. Using (\ref{dec}), we obtain the following decomposition
for $T$:
\be\label{dect}T=-2J\theta\we\o-\J\O_0.\ee
Therefore, the Lee form $\theta$ is parallel (and not
zero). To keep the notation simple, and also because in a LP manifold
the Lee form itself is not so relevant as the complex line generated
by it, we will write
$$T=\eta\we\o-\J\O_0,$$
with $\eta:=-2J\theta$, and recall that the Lee form of $\o$ can be
retrieved as $\frac{1}{2}J\eta$. 
We will also denote by $\eta$ the
dual vector field, and, by our convention in Section 2, $J\eta$ (as a
vector field) coincides with the $1$--form denoted by $J.\eta$.



In order to study the algebraic structure of $T$, we will use two
facts about a characteristic connection with parallel torsion: the
first one is the (algebraic) Bianchi identity, which follows from 
$$d^\nb I=T,$$
where $I$ is the identity of $TM$, seen as a 1-form with values in
$TM$, and $d^\nb$ is the exterior differential of a $k$-form with
values in $TM$:
\be\begin{split}d^\nb\a(X_0,\dots,X_k):=&
  \sum_{i=0}^k(-1)^i\nb_{X_i}\left( \a(X_0,\dots,\hat{X_i},
\dots,X_k)\right)+\\
&+\sum_{i<j}(-1)^{i+j}\a([X_i,X_j],\dots,\hat{X_i},\dots,
\hat{X_j},\dots).\end{split}\ee
It is well known that
$$d^\nb d^\nb\a=R^\nb\we\a,$$
where $R^\nb\we\a$ is the exterior product of the 2-form $\R^\nb$
(with values in $\End(TM)$) with the $k$-form $\a$.

Therefore
$$R^\nb\we I=d^\nb d^\nb I=d^\nb T,$$
more precisely
$$R^\nb_{X,Y}Z+R^\nb_{Y,Z}X+R^\nb_{Z,X}Y=T(T(X,Y),Z)+T(T(Y,Z),X)+T(T(Z,X),Y).$$
The left hand side is the Bianchi expression in the curvature $R^\nb$,
and the right hand side is -- in our case, where $T$ is totally
skew-symmetric -- a $4$-form $\Omega$, defined as
\be\label{Omg}\Omega:=\frac{1}{2}\sum_{i=1}^{2m}T(\e_i,\cdot,\cdot)\we
T(\e_i,\cdot,\cdot),\ee
where $\{\e_i\ |\ 1\le i\le 2n\}$ is an orthonormal basis of $TM$. Thus
\be\label{bia}g(R^\nb_{X,Y}Z+R^\nb_{Y,Z}X+R^\nb_{Z,X}Y,V)
=\Omega(X,Y,Z,V).\ee 
Note that, if we compute the exterior differential of the $3$--form
$T$ using a formula similar to (\ref{dom}), we obtain
$$dT=2\O.$$

The second fact that we will use to decompose $T$ involves the
$\nb$-codifferential of a $k$-form: 
$$\dl^\nb\a:=-\sum_{i=1}^{2n}\e_i\lrcorner\nb_{\e_i}\a.$$
This codifferential is obviously zero on a $\nb$-parallel form (fact which
does not hold for $d^\nb$ defined above). The $\nb$-codifferential is
related to the usual $g$-codifferential (defined as $\dl:=-*d*$ on an
even-dimensional manifold), as shown in \cite{af}:
\begin{prop}\label{codif} {\rm \cite{af}} Let $(M,g)$ be a
  Riemannian manifold and $\nb$ a metric connection with
  skew-symmetric torsion $T$. Then
$$\dl^\nb\a-\dl\a=\frac{1}{2}\sum_{i<j}(\e_i\lrcorner\e_j\lrcorner\a)
\we(\e_i\lrcorner\e_j\lrcorner T),$$
for any $k$-form $\a$, $k>1$. In particular $\dl^\nb T=\dl T=0$.\end{prop}

Assuming $(M,g,J)$ is a hermitian manifold with
characteristic connection $\nb$, we can now state the following

\begin{prop}\label{struct} Let $(M,g,J)$ be a $2n$-dimensional hermitian manifold
  such that the torsion $T$ of the characteristic connecton $\nb$ is
  $\nb$-parallel. Assume, moreover, that the Lee form
  $\theta:=\frac{1}{2}J\eta$ of $\o$ is not zero.  
  Then
\be\label{decomp}T=\eta\we\o_+ +\J\eta\we\o_-+T_0,\ee
where\bi
\item $\eta$ is a $\nb$-parallel form dual to a Killing vector field of
constant length, which is moreover real-holomorphic (therefore
$[\eta,J\eta]=0$); denote by $E:\R\eta\op\R J\eta$ be the complex
(integrable) distribution spanned by the Lee vector field and let $H$
be its orthogonal complement; 
\item $\o_+:=d\eta$ ,$\o_-:=d(J\eta)$ are $(1,1)$-forms on
$H$ and
\item $T_0$ is a trace-free $(2,1)+(1,2)$--form on $H$.\ei
 Moreover, if we denote by $A_+,A_-$ the
corresponding $J$-invariant skew-symmetric endomorphisms of $H$, then 
$$[A_+,A_-]=0\mbox{ and } A_\pm.T_0=0,$$
where
 $A.\a(X_1,\dots,X_k):=-\sum_{i}\a(\dots,AX_i,\dots).$
\end{prop}
\begin{proof}
The decomposition (\ref{decomp}) follows directly from the parallelism
of $\theta$ and of the splitting $TM=E\op H$. It is also clear that 
$$d\eta=T(\eta,\cdot,\cdot),$$
for the $\nb$-parallel form $\eta$, which implies that $\o_+:=d\eta$
and $\o_-:=d(J\eta)$. It remains to show that they are $(1,1)$--forms,
the commutation relations, and that $\eta, J\eta$ are Killing and
real-holomorphis (i.e., their flow preserves $J$). 

That $\eta$ and $J\eta$ are Killing follow from the fact that they are
$\nb$--parallel and the torsion $T$, relating $\nb$ to the Levi-Civita
connection $\nb^g$, is skew-symmetric, thus
$\nb^g\eta$ and $\nb^g(J\eta)$ are skew-symmeric endomorphisms of
$TM$, identified with some multiple of $A_+$, resp. $A_-$. That their
flow preserves $J$ is thus equivalent to the $J$-invariance of
$A_\pm$, or to $\o_\pm$ being of type $(1,1)$.

To prove this, we use Proposition \ref{codif}. As we have seen before,
$$\dl T=0.$$
On the other hand,
$$\dl T=-*d*T,$$
and we want to use the expression of $*$ related to the operators
induced by $J$, as given in Proposition \ref{ide}. For this, we write 
\be\label{ddd}T=\eta\we\o-\J\O_0,\ee
with
$\O_0\in\L^{(2,1)+(1,2)}_0M$, and we get
$$*(-\J
\O_0)=\O_0\we\frac{\o^{m-3}}{(m-3)!},$$
and
$$*(\eta\we\o)=J\eta\we\frac{\o^{m-2}}{(m-2)!}=\frac{1}{m-2}
  (J\eta\we\o)\we\frac{\o^{m-3}}{(m-3)!},$$ 
and recall that $d\o=\J T=J\eta\we\o+\O_0$. We obtain
$$*T=d\o\we\frac{\o^{m-3}}{(m-3)!}-\frac{m-3}{m-2}(J\eta\we\o)
\we\frac{\o^{m-3}}{(m-3)!}.$$ 
WWe know that $*T$ is closed, so we have that
$$d\left((J\eta\we\o)
\we\frac{\o^{m-3}}{(m-3)!}\right)=0.$$
But this means that
$$-*d*(\eta\we\o)=0,$$
so both components of $T$ in (\ref{ddd}) have vanishing
codifferential. They are also $\nb$-parallel, so
$\dl^\nb(\eta\we\o)=0$ as well. Applying Proposition \ref{codif}, we
get that
$$\a:=\sum_{i<j}(\e_i\lrcorner\e_j\lrcorner(\eta\we\o))
\we(\e_i\lrcorner\e_j\lrcorner T)=0.$$
But
$$(\eta\we\o)(\e_i,\e_j,X)=-\eta(\e_i)g(\e_j,JX))+\eta(\e_j)g(\e_i,JX)+
  \o(\e_i,\e_j)\eta(X).$$
Suppose that $J\e_{2i-1}=\e_{2i}$ and that $\e_1,\e_2\in E$ and the
rest spans $H$.
In order to evaluate $\a$ on various arguments $X,Y$, we compute first
\be\begin{split}&\sum_{i,j=1}^{2m}(\eta\we\o)(\e_i,\e_j,X)T(\e_i,\e_j,Y)=\\
&=\sum_{i,j=1}^{2m}\left(\eta(X)\o(\e_i,\e_j)-
  \eta(\e_i)g(JX,\e_j)+\eta(\e_j)g(\e_i,JX)\right) T(\e_i,\e_j,Y)\\
&=2\tr_\o T(Y)\eta(X)-T(\eta,JX,Y)+T(JX,\eta,Y).\end{split}\ee
We obtain thus
$$\a(X,Y)=(\tr_\o T\we \eta )(X,Y)-T(\eta,JX,Y)-T(\eta,X,JY)=0.$$
We know that $\tr_\o T$ is colinear to $\eta$, thus we obtain that
$T(\eta,\cdot,\cdot)$ is a $(1,1)$--form. But this means that
$\o_+\in\L^(1,1)M$, as required. By $\J T=J.T$ it also follows that
$\o_-\in\L^(1,1)M$. Note that this already implies that
\be\label{Txx}T(\eta,J\eta,\cdot=0).\ee

The commutation relations follow from (\ref{bia}); indeed, from
$\nb\eta=0$, it follows that the curvature terms $R(X,Y,Z,\eta)$ vanish, thus
$$\O(X,Y,Z,\eta)=0,\ \forall X,Y,Z\in TM.$$
From (\ref{Omg}), we obtain, by setting $Z:=J\eta$,
\be\label{omeg}\begin{split}0=\O(X,Y,J\eta,\eta)=\sum_{i=1}^{2m}&\big{(}
T(\e_i,X,\eta)T(\e_i,Y,J\eta)-T(\e_i,Y,\eta)T(\e_i,X,J\eta)-\\
&-T(\e_i,X,Y)T(\e_i,\eta,J\eta) \big{)}\end{split},\ee
where the last term vanishes from (\ref{Txx}). But this is equivalent
to
$$\sum_{i=1}^{2m}\o_+(\e_i,\cdot)\we \o_-(\e_i,\cdot)=0,$$
and this is equivalent to $[A_+,A_-]=0$.

If we set in (\ref{Omg}) $X,Y,Z\in H$ and $V:=\eta$, (\ref{omeg})
implies, in a similar way, that
$$A_+.\O_0=0,$$
and, if $V=J\eta$, we obtain $A_-.\O_0=0,$ as required.
\end{proof}

The skew-symmetric endomorphisms $A_\pm$ can thus be
diagonalized simultaneously, thus we decompose $H$ in the eigenspaces
$H_i$, $i=1,\dots,k$, of these endomorphisms, such that
$$A_\pm=a_i^\pm J\ \mbox{ on }\ H_i.$$
Of course, the decomposition 
$$TM=E\op\op_{i=1}^kH_i$$
is orthogonal and $\nb$-parallel, thus multiplying the metric $g$ with
some constants on each $H_i$, and even changing the complex structure
on $E$ (see below) will produce a new $\nb$-parallel
hermitian structure $(g',J')$. However, even though the torsion $T$ of $\nb$ is
unchanged, the torsion {\em tensor} $g(T(\cdot,\cdot),\cdot)$ will
change, and it is not necessarily skew-symmetric any more. The
relation between the characteristic torsions of $(g,J)$ and $(g',J')$
is investigated in the following
\begin{prop}\label{pardef} Let $(M,g,J)$ a hermitian manifold with
  characteristic connection $\nb$ and parallel torsion 
$$T=\eta\we\o_++J\eta\we\o_-+T_0.$$
 Denote by $A_\pm$ the $J$-invariant skew-symmetric endomorphisms
 corresponding to $\o_\pm$. Suppose 
\be\label{ggdec}TM=E\op H=\oplus_{i=0}^k H_i\ee
is the $\nb$-parallel and $g$-orthogonal decomposition of $TM$ in
$H_0=E$, the complex line generated by the Lee form, and 
$$H:=E^\perp=\oplus_{i=1}^k H_i,$$
where $H_i$ are the eigenspaces of the (commuting)
endomorphisms $A_+,\ A_-$.

Let
$(g',J')$ be a $\nb$-parallel hermitian structure, such that 
$$J'|_H=J|_H, g'|_{H_i}=a_ig|_{H_i},\ i=1,\dots,k,$$
and such that the decomposition (\ref{ggdec}) is also $g'$-orthogonal.

 Then the following hold:\bi
\item $J'$ is integrable,
\item the characteristic connection $\nb'$, corresponding to the
  hermitian structure $(g',J')$ differs from $\nb$ by a
  $\nb$-parallel tensor $A:TM\ra\End(TM)$,\\
Suppose now that $(M,g,J)$ is LP. Then we also have
\item for any $X\in TM$, $A_X$ is
  skew-symmetric w.r.t. both metrics $g,g'$, and it commutes with both
  complex structures $J,J'$. 
\item The Lee form of
  $\o':=g'(J'\cdot,\cdot)$ is a linear combination (with constant
  coeficients) of $\xi$ and $J\xi$, where $\xi$ is the Lee form of
  $\o:=g(J\cdot,\cdot)$ 
\item the torsion $T'$ of $\nb'$ is
  parallel w.r.t. both connections $\nb,\nb'$,
\item $\xi$ and $J\xi$ are $\nb'$-parallel as well,
\ei
\end{prop}
\begin{proof} The first claim follows from that fact that all
  $\nb$-parallel sections of 
$E$ are (real parts of) holomorphic vector fields (both for $J$ and
$J'$). Indeed,
we consider the Nijenhuis tensor $N'$ of $J'$:
$$4N'(X,Y)=[J'X,J'Y]-J'[J'X,Y]-J'[X,J'Y]-[X,Y],$$
and we consider the cases $X,Y\in E$, $X,Y\in H$ and $X\in E$, $Y\in
H$. In the first case $N'(X,Y)=0$ holds because $X,Y$ are tangent to
some $2$--dimensional leaves, and every almost complex structure on
such a manifold is integrable, and in the second case, we can replace
$J'$ by $J$.

In the third case, we have $J'Y=JY$ and, because $X$ and $J'X$ are
Killing, they preserve the distribution $H$, thus all $J'$s can be
replaced with $J$s except for the two occurences of $J'X$. The
Nijenhuis tensor becomes, in this case,
$$4N'(X,Y)=\mathcal{L}_{J'X}J(Y)-J\mathcal{L}_X(Y),$$
and both terms vanish because the flows of both $X$ and $J'X$
preserve $J$.

This proves claim 1.
\medskip

For claim 2, we know that, because $J'$ is integrable, there exists a
(unique) characteristic connection $\nb'_XY=\nb_XY+A_XY$ for the
hermitian structure $(g',J')$, \cite{af}. In what follows, we compute
it explicitly:

Denote by 
\beas A(X,Y,Z)&:=&g'(A_XY,Z),\\
\tau(X,Y,Z)&:=&g'(T(X,Y),Z),\\
T'(X,Y,Z)&:=&g'(T'(X,Y),Z),\eeas
where $T(X,Y)=\nb_XY-\nb_YX-[X,Y]$ and $T'(X,Y)=\nb'_XY-\nb'_YX-[X,Y]$
are the torsions of $\nb$, resp. $\nb'$. The conditions $\nb'g'=0$,
$\nb'J'=0$ and $T'\in\L^3M$ imply:
\bea\label{As} A(X,Y,Z)+A(X,Z,Y)&=&0,\\
\label{AJ} A(X,J'Y,J'Z)&=&A(X,Y,Z),\ \mbox{and, resp. }\\
\label{A3s} A(Y,X,Z)+A(Z,X,Y)&=&\tau(X,Y,Z)+\tau(X,Z,Y),\eea
for all vectors $X,Y,Z\in TM$. These equations imply, after some
computations,
\be\label{A'}\ba{rcrcl}A(X,Y,Z)&=&\displaystyle{\frac{1}{2}\ \Big{(}\ }
  (-\tau(X,Y,Z)&+&\tau(J'X,J'Y,Z))\ +\\
&&+(\tau(Y,Z,X)&+&\tau(J'Y,J'Z,X))\ -\\
&& -(\tau(Z,J'X,J'Y)&+&\tau(J'Z,X,J'Y))\ \Big{)}.\ea\ee   
Therefore, $A$ is $\nb$-parallel (claim 2.), but, strictly speaking,
we haven't proven that the equations (\ref{A'}) imply the relations
(\ref{As}-\ref{A3s}). This implication follows by the existence of the
characteristic connection \cite{af}, or, alternatively, by the
following lemma: 


\begin{lema}\label{l1} Let $(M,g,J,\nb)$ as above, and let $(g',J')$
  be a $\nb$-parallel hermitian structure on $M$ (thus $J'$ is
  supposed integrable). Then the tensor $A$ defined in (\ref{A'})
  satisfies the relations (\ref{As}-\ref{A3s}).\end{lema}
\begin{proof} This follows by straightforward computation; each of the
  equations (\ref{As}-\ref{A3s}) turns out to be equivalent to the
  fact that the Nijenhuis tensor $N'$ becomes totally skew-symmetric
  if we use the metric $g'$, which follows from the integrability of
  $J'$.
 (Note that $g(N'(\cdot,\cdot),\cdot)$ is always totally skew-symmetric,
 because this tensor can be expressed, for
  the $\nb$-parallel $J'$, in terms of the torsion $T$, which is
  totally skew-symmetric; however, this does not necessarily imply that
  $g'(N'(\cdot,\cdot),\cdot)\in\L^3M$, which is what we want.) \end{proof}
We can thus assume that the connection $\nb'$ is indeed the
characteristic connection of the hermitian structure $(g',J')$, and
that it differs from $\nb$ by a $\nb$-parallel tensor $A$ (claim 2.).
\medskip
let us compute now the values of $A$ on specific vectors. First, let
$X\in H_i$, $Y\in H_j$ and $Z\in H_l$, for $i,j,l=1,\dots,k$. Then
$J'$ coincides with $J$ on these vectors, and the
$g'$ scalar product with the elements $Z$, $X$ and $J'Y=JY$ can be
replaced, in (\ref{A'}), by $g$ multiplied by the factors $a_l$, $a_i$
and, resp. $a_j$. We replace thus $\tau$ by some multiples of $T$ and
we obtain:
\be\label{Aijl}\begin{split}&2A(X,Y,Z)=2a_lg(A_XY,Z)=\\
&=(a_i-a_l)T(X,Y,Z)+(a_l-a_j)T(JX,JY,Z)+
(a_i-a_j)T(X,JY,JZ),\\
& \forall X\in H_i,\ Y\in H_j,\ Z\in
H_l.\end{split}\ee 
Of course, the $J$-invariance of $A_X$ is obvious for these
arguments, but if we want to check that $g(A_XY,Z)+g(A_XZ,Y)=0$, we
need to check that
$$\frac{2}{a_l}A(X,Y,Z)+\frac{2}{a_j}A(X,Z,Y)=0.$$
We use (\ref{Aijl}) and note that the first and the third terms on the
right hand side are already skew-symmetric in $Y$ and $Z$. We obtain thus
$$\frac{2}{a_l}A(X,Y,Z)+\frac{2}{a_j}A(X,Z,Y)=
\frac{a_l-a_j}{a_l}T(JX,JY,Z)+ \frac{a_j-a_l}{a_j}T(JX,JZ,Y).$$
This vanishes for any positive numbers $a_1,\dots,a_k$ iff
$T_0(X,Y,Z)=0$ each time two of the vectors $X,Y,Z$ belong to two
different eigenspaces $H_i\ne H_j$. This is satisfied by the
hypothesis that $(M,g,J)$ is LP, i.e., $T_0=0$. But in this case
\be\label{Aijl1} A(X,Y,Z)=0,\ \forall X\in H_i,\ Y\in H_j\ Z\in\
H_l.\ee
The other case is when $X,Y$ or $Z$ belongs to $E$. (If two of them
belong to $E$, all the terms in (\ref{A'}) vanish). Let $Z\in E$,
$X\in H_i$ and $Y\in H_j$, with $i,j\ge 1$. We replace $J'X$ and $J'Y$
by $JX$, resp. $JY$, and the first line of the right hand side of
(\ref{A'}) vanishes. We obtain thus
\be\label{Aij0}2A(X,Y,Z)=a_i(T(Y,Z,X)+T(JY,J'Z,X))
-a_j(T(Z,JX,JY)+T(JZ,X,JY)),\ee
$$ \forall X\in H_i,\ Y\in H_j.$$
We use again that $T(Z,\cdot,\cdot)$ is $J$--invariant for any $Z\in
E$, thus
$$2A(X,Y,Z)=(a_i-a_j)(T(X,JY,J'Z)+T(X,Y,Z)).$$
Recall that $T(X,Y,Z)=0$, $\forall Z\in E$ and $X,Y$ eigenvectors of
$A_\pm$ for {\em different} eigenvalues. Thus $A(X,Y,Z)=0$ if $i\ne
j$, therefore 
$$A(X,Y,Z)=0\ \forall Z\in E.$$
Because of the skew-symmetry of $A$ in the last two arguments, it
follows equally that
$$A(X,Y,Z)=0\ \forall Y\in E.$$
It remains thus to compute $A(X,Y,Z)$ with $X\in E$. Let $Y\in H_j$
and $Z\in H_l$, $j,l\ge 1$. Using
(\ref{A'}), after some similar computations, we get
\be\label{A0ij} A(X,Y,Z)=a_lg(A_XY,Z)=\left(
  g'-\frac{a_j+a_k}{2}\right)\big{(}T(Y,Z),X\big{)}+
\frac{a_l-a_j}{2}T(J'X,JY,Z),\ee
$$ \forall X\in E,\ Y\in H_j,\ Z\in
H_l.$$
The $J$--invariance of $A_X$, for $X\in E$, follows directly, so it
remains to check that $A_X$ is also $g$-skew-symmetric. Note first
that, as we saw earlier, the terms $T(Y,Z)$ and $T(JY,Z)$ have no
component in $E$ if $Y$ and $Z$ belong to different eigenspaces
$H_j\ne H_l$ of $A_\pm$. Thus
\be\label{A0jl} A_XY\in H_j,\ \forall X\in E,\ Y\in H_j.\ee
We suppose thus $Y,Z\in H_j$ and we obtain
$$g(A_XY,Z)=\left(\frac{2}{a_j}g'-2g\right)\big{(}T(Y,Z),Z\big{)},$$
which is skew-symmetric in $Y,Z$. This proves Claim 3.
\medskip
 
We have shown (using the LP condition) that $A_XY=0$ $\forall X,Y\in
H$. Therefore,
$$T'(X,Y,Z)=g'(T(X,Y),Z),\ \forall X,Y\in H,$$
moreover $T'(X,Y,Z)=0$ if $X,Y,Z\in H$ (because $M$ is
LP). Therefore, the trace of $T'$ w.r.t. $\o'$ is in the dual space
to $E$, which proves the fourth claim. Note that it is possible that
this trace vanishes for some choice of the constants $a_1,\dots,a_k$.
\medskip

Because $A$ (and $T$) is algebraically expressed by $A_\pm$ and $\xi$ and
$J\xi$, and because $[A_+,A_-]=0$, it follows that $A_X.A=0$ and
$A_X.T=0$, $\forall X\in E$, therefore
$$\nb'_X A=\nb_XA+A_X A=0$$
and then also $\nb'_XT=0$, $\forall X\in E$. For $X\in H$
$\nb'_X=\nb_X$, therefore $A$, $T$ and all linear combinations (with
constant coefficients) of $\xi$ and $\J\xi$ are $\nb'$--parallel as
well (Claims 5 and 6).
\end{proof}
\begin{defi} The modification $(g',J')$ of the GCE structure
  $(g,J)$, as in the Proposition \ref{pardef}, is called {\em parallel
    modification}. A curve $t\mapsto (g_t,J_t)$, with
  $(g_0,J_0)=(g,J)$, of {\em parallel  modified} GCE structures
  is called a {\em parallel deformation of GCE structures}.\end{defi} 
\begin{rem}\label{remk}
The $\nb$-parallelism of $J'$, for a hermitian structure with parallel
characteristic torsion, does not even imply, in
general, that it is integrable, as an example from twistor theory
shows \cite{nk}: indeed, for a compact nearly K\"ahler
$6$-manifold $(M,g,J)$ with {\em reduced} characteristic holonomy
(i.e., there is a $\nb$-invariant splitting $TM=E\op H$ in complex
subbundles), it follows that $(M,g,J)$ is the twistor space of $S^4$
or of $\cp^2$, for the non-integrable almost complex structure $J$
\cite{nk}. By exchanging $J$ by $-J$ on the twistor fibers, we obtain
an integrable (in fact, it is even
K\"ahler w.r.t. to a rescaled metric $g'$) complex structure on
$M$. If we do not change the metric $g$, the characteristic connection $\nb$
for the hermitian structure $(g,J)$ is still characteristic for
$(g,J')$, and its torsion is still parallel (for a nearly K\"ahler
manifold $(M,g,J)$, the characteristic torsion is always parallel,
\cite{kir}, \cite{nk}). The $\nb$-parallel
modification $J$ of $J'$ is, of course, not integrable. 
The point here is that $(M,g,J')$ is hermitian, but with
vanishing Lee form. It is thus essential that the modification of the
hermitian structure $(g,J)$ is made regarding of the Lee distribution
$E$, in order to obtain the results in Proposition \ref{pardef}.
\end{rem}

We can use the modification $(g',J')$ of the hermitian GCE metric
$(g,J)$ to prove that the Calabi-Eckmann complex structures $J_a$
admit GCE hermitian structures (Proposition \ref{ce-par}). Note that the
Lee form doesn't change if all $a_i=1$ (even if $J'\ne J$).

Another application of Proposition \ref{pardef} is to determine the
local structure of a GCE manifold provided $k\le2$, i.e., there are
at most two common eigenspaces of $A_\pm$; the case where there is only one
eigenspace $H$ is the Vaisman case: indeed, in this case $\o_-=0$ and $\o_+$
must be a multiple of $\o|_H$. If $k=2$, the following cases occur:
\begin{cor}\label{caz} Let $H:=H_1\op H_2$ be the (non-trivial)
decomposition of $H=E^\perp$ in common eigenspaces of the
endomorphisms $-J\circ A_\pm$ corresponding to the exact
$(1,1)$--forms $\o_+=d\eta$ and $\o_-=d(J\eta)$ in Proposition
\ref{pardef}, and let $a^\pm_1\ge a^\pm_2$ be the corresponding eigenvalues. Then
the following situations occur:\bi
\item $\o_-=0$, $a^+_2<0<a^+_1$;
\item $\o_-=0$, $a^+_2=0<a^+_1$;
\item $\o_-=0$, $0<a^+_2\le a^+_1$;
\item $\o_-\ne 0$.\ei
Then, by an appropriate choice of hermitian structure $(g',J')$ as in
Proposition \ref{pardef}, we obtain the following structures:\bi
\item In cases 1. and 3., $(M,g',J')$ is a (pseudo)-Vaisman manifold
  (In case 1., the metric has mixed signature: it is negative-definite
  on $H_2$);
\item In case 2.,  $(M,g',J')$ is locally the product of a Sasakian
  manifold, a real line and a K\"ahler manifold;
\item In the (generic) case 4., $(M,g',J')$ is locally the product of
  two Sasaki manifolds.\ei
\end{cor}
\begin{proof}
If $\o_-=0$, then the constants $a_i$ will be chosen such that the
eigenvalues of $\o_+$ are $-1,0$ or $1$ w.r.t. $\o'$ (we put
$J':=J$). In this case, the metric $g'$ (and actually $g$ itself)
clearly has closed Lee form. In Case 3., the metric is already
Vaisman, and in case 1., the pseudo-hermitian metric
$\o^s:=\o'|_E+\o'|_{H_1}-\o'|_{H_2}$ is locally conformally
pseudo-K\"ahler with parallel Lee form, therefore pseudo-Vaisman. 

If $\o_-\ne 0$, then $\a_1>0>\a_2$, since the trace of $\o_-$
w.r.t. $\o$ is zero. On the other hand, the eigenvalues $a^+_1, a^+_2$
are not both zero, thus the vectors $a^+:=(a^+_1,a^+_2)$ and
$a^-:=(a^-_1,a^-_2)$ form a basis in $\R^2$. There exists a unique
invertible $2\x 2$ real matrix 
$$R:=\left(\ba{cc}p&q\\ r&s\ea\right)$$
such that $R(2,0)=a^+\in\R^2$ and $R(0,2)=a^-$. We let $R$ act on
$E\simeq \R^2$ set $\eta':R\eta$ and $J'\eta':=R(J\eta)$. This defines the
complex structure $J'$ and we set $g'$ on $E$ such that $\eta'$ and
$J'\eta'$ are unit vectors. We can write then
$$T'(\eta',\cdot,\cdot)=d\eta'=d(R\eta)=2\o|_{H_1}$$
and
$$T'(J'\eta',\cdot,\cdot)=d(J'\eta')=d(R(J\eta))=2\o|_{H_2},$$
therefore the distributions $D_1:=\R\eta'\op H_1$ and $D_2:=\R
J'\eta'\op H_2$ are both integrable and the leaves of the corresponding
foliation are Sasakian. These foliations are $g'$-orthogonal to each other
and correspond to a local product structure.\end{proof}
\begin{rem} There is another way to reduce case 1. from the above
  Corollary to a Vasiman manifold (for a positive-definite metric, but
  for the opposite orientation):
  we can replace $J$ on $H_2$ by $J':=-J$ and let $J'=J$ on $H_1$ and
  on $E$; the corresponding K\"ahler form changes sign on $H_2$ and we
  obtain a GCE hermitian metric (the integrability of this $J'$ is easily
  checked in this GCE setting --- it would not necessarily be true in
  general, cf. Remark \ref{remk} above) $(g,J')$ corresponding to the
  case 3. of the   Corollary.\end{rem} 

Conversely, the product of two Sasakian manifolds is GCE with the
product metric $G$ and the corresponding complex structure $J$, but
also w.r.t. the {\em parallel modification} $(g',J')$ of the
hermitian structure $(g,J)$ from Proposition \ref{pardef}. We will
call these GCE manifolds {\em modified Sasakian products}.

\section{Compact generalized Calabi-Eckmann manifolds of dimension
  $6$} In this section, $(M^6,J,g)$ denotes a compact GCE hermitian
manifold. The torsion $T$ of the characteristic connection
$\nb$ is parallel and decomposes according to Proposition
\ref{struct} (with $T_0=0$). Of course, the number of common
eigenspaces to $A_\pm$ is at most $2$, because $H$ is complex
$2$-dimensional. We have seen in Corollary \ref{caz} that a parallel
change of hermitian structure reduces a GCE structure either to a
Vaisman threefold (if $\d(J\xi)=0$) or to a local Sasakian product (in
the generic case). 

In the generic case, we will
give a full description of the complex manifold $(M,J)$ and of the
corresponding hemitian metric $g$.

We can assume, possibly after a parallel modification, that a GCE
metric of generic type (i.e., $\o_-\ne 0$), has as universal covering
$(\tilde M,g,J)$ a Riemannian product
of two Sasakian $3$--manifolds $(N_i,g_i,\xi_i)$, $i=1,2$, where
$\xi_i$ are the corresponding unit Reeb vector fields.

\begin{rem} A Sasakian structure can be defined in many ways:\bi
\item A Riemannian metric, with a unit Killing vector field of certain
  type, plus an integrability condition;
\item A $CR$-structure admitting a Reeb vector field whose flow
  preserves it;
\item A non-vanishing vector field $\xi$, a complex (integrable)
  structure on the quotient bundle $TM/\R\xi$ which is invariant to
  the flow of $\xi$, and a compatible dual 1-form to $\xi$.\ei
\end{rem}
The latter is the one that we will use here.

\begin{defi} A {\em pre-Sasakian} structure on a $3$-manifold $M$ is a
  non-vanishing vector field $\xi$ and a complex structure $J$ on
  $TM/\R\xi$ that is invariant by the flow of $\xi$.\end{defi}
\begin{prop} A Sasakian structure on $M$ is given by a pre-Sasakian
  structure together with a $\xi$-invariant contact form $\eta$ such that
$$\eta(\xi)=1,$$
and that $d\eta(\cdot,J\cdot)$ is symmetric and positive definite on
$TM/\R\xi$.\end{prop} 
\begin{proof} Indeed, the Sasakian metric is then 
$$g:=\eta^2+\frac{1}{2}d\eta(J\cdot, \cdot).$$\end{proof}

Recall that in the classification of Saskian structures on compact
$3$-manifolds \cite{lck}, the first geometrical information that we were able to
give was regarding the {\em pre-Sasakian structure}, in particular on
the Reeb vector field $\xi$. Once the Sasakian structure was given,
it was possible to deform the contact form, by keeping the same
pre-Sasakian structure, in order to get a locally homogeneous metric
(this deformation of the $CR$ structure was called {\em of second
  type}, while the {\em deformations of the first type} consisted in
changing the metric while keeping the same underlying $CR$ structure;
globally, these deformations only occur on the sphere $S^3$ or some of
its quotiens, \cite{lck}, \cite{cr}, \cite{s3}).

By analogy we define:
\begin{defi} An {\em elliptic pre-complex structure} on a
  six-dimensional manifold $M$ is a free action of $\C$ on $M$,
  generating an integrable distribution by $2$-planes $D$, and
  a $\C$-invariant complex {\em integrable} structure $J^D$ on $TM/D$,
  i.e., $(J^D)^2=-I$ on $TM/D$. $J^D$ is
  called integrable if the Nijenhus tensor
$$4N^{J^D}(X,Y):=[J^DX,J^DY]-J^D[J^DX,Y]-J^D[X,J^DY]-(J^D)^2[X,Y]$$
has values in $D$, for any vector fields $X,Y$ on $M$. Here, $J^D$ is
extended to an endomorphism of $TM$ by being zero on $D$.
\end{defi}
In case $M$ is the total space of an elliptic holomorphic fibration
over a complex manifold, the {\em elliptic pre-complex structure}
contains the information regarding the elliptic fibers, the fibration
itself (in the smooth category), and the complex structure on the base
manifold. It is not clear if, for a given such structure, a compatible
complex structure on $M$ exists, inducing the given elliptic
pre-complex structure. In this paper, the considered elliptic
pre-complex structures are always induced by some integrable  complex
structures.

We will show that, if $(M^6,J,g)$ is a compact, generalized
Calabi-Eckmann hermitian manifold of generic type (i.e., it is locally
isometric to a {\em modified} Sasakian product, see above), then its elliptic
pre-complex structure is locally homogeneous:

\begin{thm}\label{th1} Let $(M^6,J,g)$ be a compact, generalized
Calabi-Eckmann hermitian manifold of generic type and suppose (after
parallel modification) that its universal
covering $\tl M$ is isometric to a Riemannian product of two Sasakian
$3$-manifolds $M_1$ and $M_2$. 
Denote by $\xi_i$, $i=1,2$, the
Reeb vector fields of these Sasakian structures, and let
$E:=\R\xi_1\op\R\xi_2$, and $J^E$ be the induced complex structure on
$TM/E$ (this complex structure is $\xi_i$-invariant, $i=1,2$).

Then there is another hermitian structure $(J',g')$ on $M$, which is
obtained from $(J,g)$ through an isotopy of generalized Calabi-Eckmann
structures of generic type,  and such that $(\tilde M^6,J',g')$ is a
Lie group $G$ and $(J',g')$ is a left-invariant hermitian
structure. If both $M_1,M_2$ are non-compact, then the isotopy can be
chosen to preserve the pre-elliptic complex structure.

Moreover, the Lie group $G$ is the product $G=G_1\times G_2$, where
$G_i\in\{SU(2),\tsl,\mbox{Nil}_3\}$, and the
fundamental group $\pi_1(M)$ acts on $G$ by automorphisms of the
Sasakian product (this group is $8$-dimensional).
\end{thm}
\begin{rem} In some cases (if one of the factors $G_i$ is compact), we
  can show that, after passing to a finite covering, the fundamental
  group $\pi_1(M)$ is actually a cocompact lattice of $G$, which is a
  subgroup of codimension $2$ in the group of automorphisms of the
  structure, but it is unknown whether this fact holds in
  general.\end{rem}
\begin{proof}
Consider $\tilde M$ the universal covering of $M$, and lift the metric
$g$ to $\tl M$, which becomes a complete Riemannian manifold, on which
the de Rham decomposition theorem implies
$$(\tl M,g)=(M_1,g_1)\x (M_2,g_2),$$
where $(M_i,g_i)$, $i=1,2$ are complete, {\em cocompact} Sasakian
$3$-manifolds.


Here, {\em cocompact} means that there exists a compact set
$K_i\subset M_i$ and a group of isometries $\G_i$ of $M_i$ such that 
$$M_i=\bigcup_{\c\in\G_i}\c K_i,$$
i.e., the whole manifold can be reconstructed by the action of the group
$\G_i$ of isometries, by replicating a compact set $K_i$.

Note that $\tl M$ is cocompact with respect to $\pi_1(M)$, and thus
$M_i$ is cocompact w.r.t. $\G_i$, the projection of $\pi_1(M)$ on the
factor $M_i$ in the de Rham decomposition of $\tl M$. 
Because $\tl M$ is simply connected, so are its factors $M_i$, $i=1,2$

We want to characterize the Sasakian structures on $M_i$, $i=1,2$. We
have two situations:
\bi\item $M_i$ is compact (and simply connected), hence $(M_i.g_i)$ is
one of the Sasakian structures on the sphere $S^3$, see \cite{lck},\cite{s3}. 
\item $M_i$ is not compact, but still cocompact.\ei
In the second case we cannot apply the results about the Sasakian
structures on {\em compact} $3$-manifolds in \cite{lck},
\cite{cr}, \cite{s3}, but we will retrieve them.

Consider $P_i$ the circle bundle over $M_i$ consisting of the unit
vectors in the complex distribution $H_i$, orthogonal to  the Reeb
vector field $\xi_i$. Therefore, we have the following principal fiber
bundle:
\be\label{fibr}S^1\ra P_i\ra M_i.\ee

The group $G_i$ (containing $\G_i$) of Sasakian automorphisms of
$M_i$, i.e., the group of isometries of $(M_i,g_i)$ preserving
$\xi_i$, can thus be realized as the group of automorphisms of the
canonical parallelization of $P_i$ \cite{kob}, \cite{ball}. $G_i$ is
thus a Lie group and it can be realized as a closed (embedded)
submanifold of $P_i$ by any of its orbits, \cite{ball},
\cite{kob}. Moreover, the action of $G_i$ on $P_i$ is free and proper,
thus $G_i\backslash P_i$ is a manifold, basis of a $G_i$-principal
bundle (of which the total space is $P_i$).

It is equally well-known \cite{ball}, \cite{kob}, that a subgroup of
$G_i$ is closed iff its orbits in $P_i$ are closed. This, in return,
is equivalent to its orbits in $M_i$ being closed, because the fibers of
$P_i$ over $M_i$ are compact.

We consider the group $\Phi_i$ generated by the flow of $\xi_i$. The
closure of $\Phi_i$ in $G_i$ is either\bi
\item $\Phi_i$ itself, $\Phi_i$ being then a circle or a real line;
\item an abelian group of dimension at least $2$\ei
We begin with the second case: as the only abelian Lie groups that
admit a dense subgroup of dimension $1$ are tori, we conclude that we
have the following principal fibration, with fiber $T^k$, the
$k$-dimensional torus that is the closure in $G_i$ of $\Phi_i$:
\be\label{fibr1}
T^k\ra P_i\ra S_i,\ee
where $S_i$ is the $4-k$-dimensional quotient of $P_i$ by the torus
$T_i$. Note that by assumption $P_i$ is not compact, so neither is
$S_i$. This already excudes the case $k=4$, so we have only the cases
$k=2$ ($S_i$ is a non-compact oriented surface) or $k=3$ ($S_i$ is a
real line). In both cases $\pi_2(S_i)=0$.

We write now the long homotopy sequences corresponding to the
fibrations (\ref{fibr}) and (\ref{fibr1}):
$$\dots\ra 0=\pi_2(S_i)\ra\pi_1(T^k)\ra\pi_1(P_i)\ra\dots;$$
$$\dots\ra \pi_1(S^1)\ra \pi_1(P_i)\ra\pi_1(M_i)=0.$$
From the second line we obtain that $\pi_1(P_i)$ is a quotient of
$\Z$, but the first line implies that it must contain at least $\Z^k$,
$k\ge 2$, contradiction.

We have thus shown:
\begin{lema} On a complete non-compact, simply connected Sasakian
  $3$-manifold, the flow of the Reeb vector field is closed in the
  group of Sasakian automorphisms, and acts properly on the manifold.
\end{lema}

We have thus to distinguish between two cases\bi\item $\Phi_i\simeq\R$
\item $\Phi_i\simeq S^1$.\ei
In the first case, $\Phi_i$ acts freely on $P_i$, but not necessarily
on $M_i$. In the second case, the periods of the orbits of $\Phi_i$ on
$M_i$ could be, in principle, a quotient of the period of
$\Phi_i\simeq S^1$ by a natural number. The following lema shows that
this does not happen:

\begin{lema} Let $M_i$ be a non-compact, cocompact (w.r.t. the
  Sasakian automorphism group) complete Sasakian
  $3$-manifold. Then if $\phi_t$ has a fixed point, then it is the
  identity.\end{lema} 
\begin{proof}
Let $\c:=\phi_t$ and suppose $\c(x)=x$. Because $\c$ is an isometry
that fixes $x$, it 
is the identity iff its derivative at $x$ is the identity. Of course,
$\c_*\xi=\xi$, so $\c_*$ has to be a rotation in the plane $H_i$ at
$x$. If $\c_*\ne I$, then $\c(y)\ne y$ fo any $y$ close to $x$, but
not on the orbit of $\Phi_i$ through $x$. In other words, if $\c$ is
not the identity, the neigboring orbits of $\Phi_ix$ do not close at
time $t$.

We know, on the other hand, that, although $M_i$ is itself not
compact, there is a compact set $K_i$, 
sufficiently large, such that $M_i$ is the (infinite) union of sets
isometric to $K_i$. If we consider $K_i$ to contain the closed orbit
$\Phi_ix$, we conclude that there are (infintely many, hence) at least
two at time $t$ closed orbits of $\Phi_i$ on $M_i$, say $\Phi_ix\ne
\Phi_iy$. But these two circles can be joined by a minimizing geodesic
in the complete manifold $M_i$, and this minimizing geodesic must be
orthogonal to both circles. The flow of $\xi$ generates, thus, a set
of such minimizing geodesics between $\Phi_ix$ and $\Phi_iy$. It turns
out that every point on such a geodesic is a fixed point for
$\c=\phi_t$. In conclusion all the points on a geodesic through $x$,
orthogonal to $\xi_x$, are fixed by $\c$, contradiction.
\end{proof}

We apply the lemma as follows: if $\phi_t$ has a fixed point, then
$\Phi_i$ is $t$-periodic. This rules out the case when
$\Phi_i\simeq\R$ acts properly, but not freely on $M_i$. If $\Phi_i$
is a circle, then the Lemma shows that the period of $\Phi_i$ is the
minimal $t>0$ such that $\phi_t$ has a fixed point on
$M_i$. Therefore, $\Phi_i$ acts freely on $M_i$ in this case as well.

We have shown:

\begin{lema} On a complete, non-compact, cocompact, simply-connected
  Sasakian $3$-manifold, the Reeb orbits form a principal fiber bundle
  over a Riemann surface $S_i$.
\end{lema}

We want now to show that the Riemann surface is non-compact and simply
connected.

For this, we use again the long homotopy sequence for the fibration
$$\Phi_i\ra M_i\ra S_i,$$
and obtain
$$\dots\ra\pi_2(S_i)\ra\pi_1(\Phi_i)\ra\pi_1(M_i)=0.$$
If $S_i$ is not the sphere, then $\pi_2(S_i)=0$, hence
$\pi_1(\Phi_i)=0$ and $\Phi_i$ is a real line.

On the other hand, the same homotopy sequence shows that
$\pi_1(S_i)\simeq \pi_0(\Phi_i)=0$, so $S_i$ is simply connected. The
metric on $M_i$ induces a metric on $S_i$, in particular $S_i$ is a
simply-connected complex curve, therefore\bi
\item $S_i\simeq S^2$ and $\mbox{Isom}(S_i)\subset SO(3)$;
\item $S_i\simeq \C$ and $\mbox{Isom}(S_i)\subset \C\rtimes \C^*$, the
  group of complex automorphisms of $\C$;
\item $S_i\simeq \H:=\{z\in\C\ | \ \mbox{Im}z>0\}$, the upper
  half-plane in $\C$ and $\mbox{Isom}(S_i)\subset PSL(2,\R)$ the group
  of complex automorphisms of $\H$.\ei 

We want to exclude the case when $S_i\simeq S^2$; indeed, if this
holds, then $\Phi_i$ must be a line, otherwise $M_i$ is compact. But
then, $M_i$ is a trivial bundle over the sphere, and this cannot have
a Sasakian structure adapted to the fibration since the connection in
this principal bundle has a non-exact curvature (a multiple of the
volume form of $S_i$), contradiction.

Therefore $S_i$ is either biholomorphic to $\C$ or to $\H$. Since both
are contractible, $M_i$ is homotopically equivalent to
$\Phi_i$, thus the latter has to be simply connected, hence isomorphic
to $\R$.

We recall that $M_i$ admits a cocompact action by Sasakian
automorphisms. Therefore, the Riemann surface $S_i$ admits a cocompact
action by isometries, in particular by biholomorphisms. Let us denote
by $\G^S_i$ the group of (induced) biholomorphisms of $S_i$. 

\begin{lema}\label{isom} Let $N\ra S$ be an $\R$ principal bundle over a
  contractible Riemann surface $S$, endowed with a connection form
  $\l\in \L^1(N)$ whose curvature $d\l$ is minus twice the K\"ahler
  form $\o$ on $S$. Then every isometry of $S$ admits
  a lift as a Sasakian automorphism of $N$.\end{lema}
\begin{proof} $S$ being contractible, every bundle is topologically
  trivial. Let $(t,z)\in\R\x \S$ be coordinates on $N$ such that
  $\xi=\d_t$ and $\o$ the K\"ahler form on $S$. 
Let $\l=dt+d\l_0$, $\l_0\in\L^1S$, be the connection form
corresponding to the Sasakian structure, hence
$$d\l_0=-2\o.$$
Let $\c:S\ra S$ be a isometry, thus $\c^*\o=\o$. We want to construct
$\tl\c:N\ra N$, such that
$$\tl\c(t,z)=(\f(t,z),\c(z)),$$
and such that $\tl\c_*\d_t=\d_t$ and $\tl\c^*\l=\l$.
The first condition implies that $\d_t\f=1$, thus we can re-write
$$\tl\c(t,z)=(t+f(z),\c(z)),$$
and therefore
$$\tl\c^*\l=d(t+f(z))+\tl\c^*\l_0.$$
$\tl\c$ is a Sasakian automorphism iff
\be\label{dfff}df=\l_0-\tl\c^*\l_0.\ee
$f$ can be determined, for the contractible surface $S$, iff the right
hand side is a closed form. But 
$$d\l_0=\o=\tl\c^*\o=\tl\c^*d\l_0,$$
therefore $d(\l_0-\tl\c^*\l_0)=0$.\end{proof}

We study now the isometry group of the basis $S$, for the two cases:
$S=\c$ and $S=\H$. The following results are classical and their
proofs elementary. The less obvious point is the case 2. from Lemma
\ref{l3}, and this follows also, for example, from \cite{moore}.
\begin{lema}\label{l2} Let $\o=f\o_0$, with $f$ a positive function,
  and $\o_0=dz_1\we dz_2$, a hermitian 
  metric on $\C$. Then $G=\mbox{Isom}(\o)$ is a closed subgroup of
  $\C\rtimes S^1=\mbox{Isom}(\o_0)$ and has the following form:\bi
\item $\dim G=0$; then $G$ is a finite extension of a lattice $L:=\Z V\x
  \Z W$ by a subgroup $G/L$ of the lattice automorphism group (which has
  $2,4$ or $6$ elements);
\item $\dim G=1$; then $G$ is an extension of $L=\Z V\op \R W$, with
  $V,W$ linearly independent, by a subgroup $G/L$ of
  $\mbox{Aut}(L)=\{\pm I\}$;
\item $G=\C\rtimes S^1$.\ei\end{lema}

\begin{lema}\label{l3} Let $\o=f\o_0$, with $f$ a positive function
  and $\o_0=\frac{1}{z_2}dz_1\we dz_2$, a hermitian
  metric on $\H$. Then $G=\mbox{Isom}(\o)$ is a closed subgroup of
  $PSL(2,\R)=\mbox{Isom}(\o_0)$ and has the following form:\bi
\item $\dim G=0$; then $G$ is a finite extension of a fuchsian lattice
  $L:=\pi_1(S_0)$ (here $S_0:=\H/L$ is a Riemann surface of genus
  $g>1$) by a subgroup $G/L$ of the lattice automorphism group (which
  is finite);
\item $\dim G=1$; then $G$ is conjugated to 
$$G_a:=\left\{\left(\ba{cc}a^k&a^kb\\ 0&a^{-k}\ea\right)\ \left|\right.\
  k\in\Z,\ b\in\R\right\},$$
where $a>1$ is fixed;
\item $G=PSL(2,\R)$.\ei\end{lema}
Now we come to the core of the argument. We want to show that we can
modify the Sasakian structure on $M_i$, leaving it $\G_i$-invariant,
such that the new Sasakian structure is homogeneous. 
\begin{lema}\label{l0} Let $N\ra\S$ be a Sasakian manifold as above
  (with $S\simeq \C$ or $\H$). There exists a constant $c>0$ and a
  $1$--form $\a\in\L^1S$ such that \bi
\item $\c^*\a=\a,\ \forall\c\in\mbox{Isom}(S,\o)$;
\item $d\a=(c-f)\o_0$.\ei
\end{lema}
\begin{proof} We use the description of the possible isometry groups
  as given in the Lemmas \ref{l2} and \ref{l3}. In the first case of
  both lemmas, $S/L$ is a compact Riemann surface, and all
  $G$-invariant data on $S$ is equivalent to corresponding $G/L$-invariant
  data on $S/L$. The existence of an $L$-invariant $\a$ as required is
  equivalent to
$$\int_{S/L}(c-f)\o_0=0,$$
which determines $c>0$. On the other hand, if $\a\in\L^1(S/L)$
  satisfies
\be\label{dalf}d\a=(c-f)\o_0,\ee
then $\c^*\a$ also satisfies the same equation, for any $\c\in
G/L$. Therefore,
$$\a_o:=\left(\sum_{\c\in G/L}\c^*\a\right)\bigg{/}|G/L|,$$
where $|G/L|$ is the number of elements of $G/L$, is a
$G/L$--invariant $1$--form satifying (\ref{dalf}), thus $\a_0$ induces
on $S$ a $1$--form with the required properties.
\medskip

In the case 3. of both lemmas \ref{l2} and \ref{l3}, the function $f$
is constant, thus we trivially set $c:=f$ and $\a:=0$.
\medskip

The remaining case is when $\dim G=1$, and we will treat the case
$S=\C$ and $S=\H$ distinctly.

Let $S=\C$. We choose a coordinate system on $\C$ such that $V\in\R$
and $W=iw$, $w>0$ (with the notations of Lemma \ref{l2}). The
$G$-invariance of $f$ implies that $f$ depends only on $z_2$ and is
$w$--periodic. 

We write
$\a:=\a_1dz_1+\a_2dz_2$ and, because of the required $G$-invariance of
$\a$, we conclude that $\a_1,\a_2$ should depend on $z_2$ alone (and
they should be $w$ --periodic); as for the condition
(\ref{dalf}) the component $\a_2dz_2$ is irrelevant, we will set it to
be zero.

In this setting, the claim is equivalent to: there exists a constant
$c>0$ and a function $\a_1:\R\ra\R$ such that
\be\label{simplC}\ba{rcl}
\a_2(z_2+w)&=&\a_2(z_2)\\
-\a_2'&=&(c-f).\ea\ee
A solution (hence all) of the second line satisfies the first line iff 
$$c=\frac{1}{w}\int_0^wf(t)dt.$$
This proves the claim for $S=\C$ and $G=\R+iw\Z$. To prove it for the
extension of this group by $\{\pm 1\}$, we take the mean value between
a solution $\a_2$ of (\ref{simplC}) and $t\ra\a_2(-t)$ (as we
did for the case when $\dim G=0$).
\medskip

Let now $S=\H$. After choosing some appropriate coordinates $z_1+iz_2$,
$z_1\in\R$, $z_2>0$, on $\H$, we
will suppose that the group $G$ is equal to the group $G_a$ from Lemma
\ref{l3}. $f$, being $G$-invariant, depends thus on $z_2$ alone,
moreover
\be\label{f-H}f(a^2z_2)=f(z_2).\ee
At this moment, we make the change of variable 
$$z_2:=e^y$$
and set $a_0:=\ln a^2$. We re-write (\ref{f-H}):
\be\label{fy} f(y+a_0)=f(y),\ \forall y\in\R.\ee

As before, the coefficients of the $1$--form $\a$ must depend on $z_2$
alone and, using the variable $y$, the required $1$--from $\a$ has to
be $\a=\b(y)e^{-y}dz_1$, with $\b$ a function, and the $G$-invariance
of $\a$ is equivalent to 
\be\label{ay} \b(y+a_0)=\b(y),\ \forall y\in\R.\ee
The differential equation (\ref{dalf}) is equivalent to
\be\label{eqq} \b'-\b=c-f.\ee
This equation is an affine differential equation that has global
solutions on $\R$. If $\b_1,\b_2$ are such solutions, their difference
satifies the associated linear equation, hence there is a constant
$p\in\R$ such that 
$$\b_1(y)-\b_2(y)=p e^y.$$
Because $(c-f)$ is $a_0$--periodic, for every solution $\b$ of
(\ref{eqq}), $y\mapsto \b(y+a_0)$ satisfies (\ref{eqq}) as well, hence
$$\b(y+a_0)-\b(y)=p e^y,$$
for some $p\in\R$. Therefore,
$$\b_0(y):=\b(y)+\frac{p}{e^{a_0}-1}e^y$$
satisfies (\ref{eqq}) and is also $a_0$--periodic, as required. This
proves the claim for $S=\H$ and $\dim G=1$, where we note that, unlike
in the other cases, $c$ can be arbitrarily chosen.
\end{proof}
The consequence of this Lemma is that the contact form $\l_0:=\l+\a$,
together with the pre-Sasakian structure on $N$ determines another
$G$-invariant Sasakian structure on $N$, for which the automorphism
group is maximal ($4$--dimensional; in particular $(N,\xi,J,\l_0)$ is
homogeneous).

As the space of hermitian metrics is convex, the passage from $\o$ to
$c\o_0$ can be made smoothly, therefore the Sasakian structure on $N$
is isotopic (through $G$--invariant Sasakian structures, and also,
without changing the pre-Sasakian structure) to the
Sasakian structure corresponding to a metric of constant
curvature. This result is also true for $N\simeq S^3$, but there
an additional deformation step is involved, that keeps the $CR$ structure
but not the pre-Sasakian structure.

The Sasakian structures corresponding to a principal bundle over a
Riemann surface of constant curvature $\k$ are left-invariant Sasakian
structures on the Lie groups $\tsl $ (for $\k<0$), $Nil^3$ (for
$\k=0$) and $SU(2)$ (for $\k>0$) \cite{lck}, \cite{cr}, \cite{s3}.

The Sasakian automorphisms of these homogeneous Sasakian manifolds are
$4$--dimensional, and are equal to\bi
\item $\tsl\x\R/\Z$ for $\k<0$; here, the factor $\R$ is the Reeb
  flow, which at some times produces the central elements of $\tsl$,
  hence the quotient by $\Z$;
\item $Nil^3\rtimes S^1$ for $\k=0$; here, the factor $S^1$ comes
  from the rotations in $\C$, the space of the Reeb orbits; $S^1$ acts
  on the Heisenberg group $Nil^3$ by rotations on the contact plane;
\item $SU(2)\x S^1/\{\pm 1\}$ for $k>0$, where $S^1$ is the Reeb flow, and, as
  in the case $k<0$, the center $\{\pm 1\}$ of $SU(2)$ is common to it
  and the Reeb flow.\ei
\end{proof}
We have shown that the hermitian structure of a generic compact GCE
manifold of dimension $6$ (complex dimension $3$) can be deformed to a
locally homogeneous one. We also know that the universal covering $\tl
M$ is a
product of two of the three standard simply connected, complete,
homogeneous Sasakian $3$-manifolds (the Lie groups $\tsl, Nil^3$ and
$SU(2)$). The topological structure of the compact quotient $M$
depends, thus, on the group $\pi_1(M)=\G$. This group is not
necessarily a product of latices in the two factors, as the following
example shows:

{\noindent\bf Example. } Let $\c_{1,2}:\C\x \C\ra\C\x \C$ be isomorphisms of
  $\C\x\C$:
$$\c_1(z,w):=(z+1,qw),\quad \c_2(z,w):=(z+i,w+1).$$
The projection of $\c_{1,2}$ on the isometry groups of the factors is
discrete and cocompact, provided that $q$ is a root of unity of order
$3,4$ or $6$. If we consider some lifts 
$$\tl\c_{1,2}:Nil^3\x
Nil^3\ra Nil^3\x Nil^3$$
of $\c_{1,2}$ as $\xi_{1,2}$-preserving isometries of the Sasakian
product $\tl M:=Nil^3\x Nil^3$ (which can be done as in Lemma \ref{l1}), then
they generate a discrete group of GCE isometries $\G$ on $\tl M$. 

The action of $\c_{1,2}$ generates a cocompact group acting on
$\C\x\C$, because 
$$a(z,w):=\c_1\c_2\c_1^{-1}\c_2^{-1}(z,w)=(z,w+q-1)$$
and 
$$b(z,w):=\c_1^2\c_2\c_1^{-2}\c_2^{-1}(z,w)=(z,w+q^2-1),$$
thus every point $(z,w)\in\C\x\C$ can be brought, first by $c_1$ and
$c_2$, to a point $(z_0,w')$, with $|z_0|<2$, then, using $a$ and $b$,
we get to a point $(z_0,w_0)$, with $|w_0|<2$ as well.

On the other hand, it can be shown that the group $\G$ contains some
Reeb translations, such that the action of $\G$ on $\tl M$ is
cocompact as well.

This latter claim follows from
\begin{lem}\label{comC} The commutator of two translation lifts
  $\tl\t_V,\tl\t_W$ in the Sasakian automorphism group of $N\ra\C$ is zero iff $V,W$ are linearly dependent. More
  precisely, if the metric on $\C$ is flat (hence $N\simeq Nil^3$),
  the commutator of $\tl\t_V,\tl\t_W$ is a Reeb translation by a
  shift that is proportional with the area of the parallelogram
  generated by $V$ and $W$. 
\end{lem}
\begin{proof}
Note that every vector field $X$ on $\C$ has a unique horizontal lift
to $N$: it is the vector field $\tl X\in H$ (where $H$ is the
orthogonal space to the Reeb field) that projects on $X\in T\C$. The flow of a translation $\t_V$ can therefore be lifted to the
flow $\phi^V$ of a horizontal vector field $\tl V\in H_i$. Note that this flow
$\phi^V$ does not preserve the contact structure on $M_i$ (in fact
$\L_{\tl V}\l=-2\tl{JV}\ne 0$), but
it preserves the pre-Sasakian structure. However, if $\t_V\in
\mbox{Isom}(\o)$, Lemma \ref{l1} implies that there exists a lift
$\tl\t_V$ acting on $N$ by Sasaian automorphisms, i.e., in particular,
$\tl\t_V^*\l=\l$. We can thus assume (after possibly composing
$\tl\t_V$ with some element in $\Phi$, the group of Reeb flows) that
$\tl\t_V(x_0)=\phi^{\tl V}_1(x_0)$, for some (hence for all) $x_0\in
p_i^{-1}(0)$. Here $\phi^{\tl V}_t$ is the flow of the vector feld $\tl V$ at time
$t$. 

We can use the replacement $\tl\t_V$, resp. $\tl\t_W$ (defined
such that $\tl\t_W(x_0)=\phi^{\tl
  W}_1(x_0)$, $\forall x_0\in p_i^{-1}(V)$), because 
the composition with a central element in $\Phi_i$ does not change the
commutator of two elements.

It follows that
$$\tl\t_W\circ\tl\t_V(x_0)=\phi^{tl W}_1\circ\phi^{\tl V}_1(x_0),$$
and we want to determine
$(\tl\t_V)^{-1}\circ\tl\t_W\circ\tl\t_V(x_0)$. Note that all elements
of $G_i$ that project on translations of $\C$ (and thus commute with
the constant vector fields $V$ and $W$) preserve the lifted
vector fields $\tl V$, $\tl W$. Therefore,
$(\tl\t_V)^{-1}$ sends the integral curve of $\tl W$ containing
$\tl\t_V(x_0)$ to the integral curve of $\tl W$ through $x_0$. We
conclude that 
$$(\tl\t_V)^{-1}\circ\tl\t_W\circ\tl\t_V(x_0)=\phi^{\tl
  W}_{-1}(x_0).$$
We suppose now that $\tl\t_W\circ\tl\t_V=\tl\t_v\circ\tl\t_w$. This
implies that $\tl\t_W\circ\tl\t_V(x_0)$ and
$(\tl\t_V)^{-1}\circ\tl\t_W\circ\tl\t_V(x_0)$ are connected by an
integral curve of $\tl V$, the image through $(\tl\t_W)^{-1}$ of the
integral curve of $\tl V$ through $x_0$.

This means that the parallelogram $P$ in $\C$ through $0,V,W,V+W$ lifts to
a closed horzontal curve $C$ in $M_i$, thus there is a local section $\s$
of $M_i\ra \C$ such that $\s(\d P)=C$. However, by Stokes' Theorem, we
have
$$\int_C\l=\int_{\d P}\s^*\l=\int_P\s^*d\l.$$
The left hand side is zero because $\l$ vanishes on the horizontal
curve $C$, and the right hand side is the area of $P$ for the metric
induced on $\C$ by the Sasakian metric on $M_i$, contradiction.

If the sequence of lifted segments does not close, we artificially
close it by adding a piece of Reeb orbit. Then, using Stokes yields
the desired result.
\end{proof}
 
We see that the groups acting freely, properly discontinuously and with
compact quotient on $G_1\x G_2$, for $G_i\in\{\tsl,Nil^3,SU(2)\}$ do
not necessarily preserve the leaves of the two orthogonal
foliations. Their general structure deserves further research.


\begin{thebibliography}{99}
\bibitem{af} I. Agricola, T. Friedrich, {\it On the holonomy of connections
    with skew-symmetric torsion}, Math. Ann. 328 (2008), 711-748.
\bibitem{afs} B. Alexandrov, T. Friedrich, N. Schoemann, {\it Almost
    Hermitian 6-manifolds revisited}, J. Geom. Phys. 53 (2005), 1–30. 
\bibitem{ball} W. Ballmann, {\em Geometric Structures}, Lecture Notes,
  UBonn 2000. 
\bibitem{lck} F. A. Belgun, {\it On the metric structure of non-Kähler
    complex surfaces}, Math. Ann. 317 (2000), 1–40.
\bibitem{cr}  Belgun, {\it Normal CR structures on compact
    3-manifolds}, Math. Z. 238 (2001), 441–460. 
\bibitem{s3} F. A. Belgun, {\it Normal CR structures on $S^3$},
  Math. Z. 244 (2003), 125–151. 
\bibitem{nk} F.A. Belgun, A. Moroianu, {\it Nearly K\"ahler metrics
    with reduced holonomy}, Ann. Glob. Anal. Geom. 
\bibitem{ce} E. Calabi, B. Eckmann, {\it A class of compact complex
    manifolds which are not algebraic}, Ann. of Math., 58 (1953),
  494–500.
\bibitem{do} S. Dragomir, L. Ornea, {\it Locally conformal K\"ahler
    geometry}, Birkh\"auser 1998.
\bibitem{pg} P. Gauduchon, {\it La 1-forme de torsion d'une
    vari\'et\'e hermitienne compacte} , Math. Ann. 267 (1984),
  495-518.
\bibitem{kir} V. Kirichenko, {\it K-spaces of maximal rank},
  Mat. Zam. 22 (1977), 465-476. 
\bibitem{kob} S. Kobayashi, {\it Transformation groups in differential
    geometry}, Ergebnisse der Mathematik und ihrer Grenzgebiete, Band
  70. Springer-Verlag, New York-Heidelberg, 1972.
\bibitem{moore} C. C. Moore, {\it Cocompact subgroups of semisimple
    Lie groups}, J. Reine Angew. Math. 350 (1984), 173-177.
\bibitem{ov}  L. Ornea, M. Verbitsky, {\it Locally conformal K\"ahler manifolds with potential}, Math. Ann. 348 (2010), 25-33.
\bibitem{sch} N. Schoemann, {\it Almost Hermitian structures with
    parallel torsion}, J. Geom. Phys. 57 (2007), 2187-2212.
\bibitem{vais} I. Vaisman, {\it Locally conformal K\"ahler manifolds
    with parallel Lee form}, Rendiconti di Matem., Roma, 12 (1979), 263-284.

  
\end{thebibliography}
\end{document}